\documentclass[10 pt, draft]{amsart}
\usepackage{amsmath}
\usepackage{amsthm}
\usepackage{amsfonts}
\usepackage{amssymb}
\newtheorem{theorem}{Theorem}[section]

\theoremstyle{definition}

\newtheorem{corollary}[theorem]{Corollary}
\newtheorem{proposition}[theorem]{Proposition}
\newtheorem{example}[theorem]{Example}
\newtheoremstyle{remark}{9pt}{9pt}{}{0pt}{\bf}{.}{0.5em}{}
\theoremstyle{remark} \newtheorem{remark}[theorem]{Remark}

\def\id{\operatorname{id}}
\def\dist{\operatorname{dist}}

\title{Holomorphic mappings preserving Minkowski functionals}
\author{\L ukasz Kosi\'nski}
\address{Institute of Mathematics, Jagiellonian University, \L ojasiewicza 6, 30-348 Krak\'ow, Poland}
\email{{lukasz.kosinski}@gazeta.pl}
\thanks{The work is partially supported by the grant of the Polish Minister for Science and Higher Education
No. N N201 361436.
Moreover, some part of the paper was prepared while the stay of the author at the ESI Institute in Vienna.}
\begin{document}
\maketitle

\begin{abstract} We show that the equality $m_1(f(x))=m_2(g(x))$ for $x$ in a neighborhood of a point $a$ remains valid for all $x$ provided that $f$ and $g$ are open holomorphic maps, $f(a)=g(a)=0$ and $m_1,$ $m_2$ are Minkowski functionals of bounded balanced domains. Moreover, a polynomial relation between $f$ and $g$ is obtained. Next we generalize these results to bounded quasi-balanced domains.

Moreover, the main results of \cite{Ber-Piz} and \cite{Bou} are significantly extended and their proofs are essentially simplified.
\end{abstract}

\section{Introduction and statement of result}

Consider the following natural problem:

Let $V,$ $U$ be neighborhoods of $0,$ $V\subset U$. Let $f,g:U\to \mathbb C^k$ be
open mappings such that $f(0)=g(0)=0$ and $||f(x)||=||g(x)||$ on $V.$ Does it follow that $||f(x)||=||g(x)||$
for all $x\in U$? Is it possible to establish any relation between $f$ and $g$? For example, if $f,$ $g$ are
biholomorphic then, by the theorem of Cartan, $f=Lg$ for some linear $L.$

\medskip
The main goal of the paper is to give an affirmative answer to these questions in more general settings (e.g. instead of norms we consider quasi-Minkowski functionals of bounded quasi-circular domains).

As a by-product of our considerations we obtain a significant generalization of the main theorem of \cite{Ber-Piz}. It seems to be interesting that we prove much stronger results without using advanced tools like the theorem of Fornaess and Sibony (see \cite{For-Sib}) which was of the key importance in the paper of Berteloot and Patrizio, and what follows we do not use currents at all - this allows us to deal with non-plurisubharmonic Minkowski functionals (see Remark~\ref{ber}). Our proof is quite elementary - the key point relies upon the investigation of the Shilov boundaries of bounded balanced domains.

Finally, applying our results we show how to extend easily the central result of \cite{Bou}.

Our main result is the following:
\begin{theorem}\label{main} Let $m_1$ and $m_2$ be Minkowski functionals of bounded balanced domains in $\mathbb C^m$ and let $U$ be a domain in $\mathbb C^k$, $k\geq m.$  Let $f,g:U\rightarrow \mathbb C^{m}$ be holomorphic mappings such that $f(a)=g(a)=0$ and $f$ and $g$ are open in a neighborhood of $a$ for some $a\in U.$ Let $q\in\mathbb R.$ Assume additionally that $m_1(f(x))= (m_2 (g(x)))^q $ for $x$ in some neighborhood $V\subset U$ of $a$.

Then $q$ is a positive rational number and:
\begin{enumerate}
\item[1)] $m_1\circ f(x)=(m_2 \circ g(x))^q$ for all $x\in U,$
\item[2)] $f$ and $g$ are related in the following sense: there is a $p\in\mathbb N$ and there are homogenous
polynomials $\xi_k$ of degree $kq$, $k=1,\ldots,p,$ (if $kq\notin \mathbb N$, then $\xi_k\equiv0$) such that \begin{equation}
\label{ex}f(x)^p+f(x)^{p-1}\xi_1(g(x))+\ldots+\xi_p(g(x))=0,\quad x\in U.\end{equation} \end{enumerate} \end{theorem}

Let us explain the notation in the above theorem. First of all recall that a~mapping $f$ is said to be \textit{open in a neighborhood of $a$} if there is a neighborhood of $a$~such that the restriction of $f$ to this neighborhood is open. For $z,w\in\mathbb C^{n}$ put $z\cdot w=(z_1w_1,\ldots,z_{n}w_{n});$ $z^k$, $k\in \mathbb Z,$ is understood analogously (i.e. $z^k:=z\cdot \ldots \cdot z,$ $z^{-1}=(z_1^{-1},\ldots,z_n^{-1}$)).

Moreover, the unit disc in the complex plane is denoted by $\mathbb D$ and $\partial_s \Omega$ stands for the Shilov boundary of a bounded domain $\Omega$ in $\mathbb C^n.$

\section{Proof of the main theorem, remarks and examples}

\begin{proof}[Proof of Theorem~\ref{main}]
Losing no generality we may assume that $a=0$ and $m\geq 2.$ Moreover, it is clear that $q\in \mathbb Q_{>0}.$ Take $p_1,p_2\in \mathbb N$ such that $q=\frac{p_1}{p_2}.$

\textit{Step 1'} First we focus our attention on the case when $k=m$. It follows from the Remmert's theorem (see \cite{Rem}) that $0$ is an isolated point of $g^{-1}(0)$ and $f^{-1}(0).$ Therefore, shrinking $V$ if necessary we may assume that $f|_V$ is proper onto image. Moreover, there is a domain $V'$ such that  $0\in V'\subset V$, $g|_{V'}$ is also proper onto image and $g^{-1}(0)\cap V'= \{0\}.$ Put $\mathcal V=g(\{x\in V':\ \det g'(x)=0\})$ and fix $\delta>0$ such that $\Omega_2= \{x \in \mathbb C^m:\ m_2(x)< \delta\}$ and $\Omega_1= \{x\in \mathbb C^m:\ m_1(x)<\delta^q \}$ are relatively compact in $g(V')$ and $f(V)$, respectively. Since $V'\cap g^{-1}(0)=\{0\}$, one can see that $g^{-1}(\Omega_2)$ is a domain.

Take $x_0\in \partial_s \Omega_2\setminus \mathcal V$ and let $G_j,$ $j=1,\ldots, p,$ be local inverses to $g|_{V'}$ defined in a~neighborhood of $x_0$, i.e. $g^{-1}=\{G_1,\ldots, G_p\}.$ It follows from the invariance of the Shilov boundary under proper holomorphic mappings (see \cite{Kos}, Theorem~3) that there is an index $i$ (fixed from now on) such that $G_i(x_0)\in \partial_s g^{-1}(\Omega_2).$ Put $y_0:=f(G_i(x_0)).$ Since $g^{-1}(\Omega_2)=f^{-1}(\Omega_1)$ we may apply the argument from \cite{Kos} again to state that $y_0\in \partial_s \Omega_1.$

We aim at showing that the map $$t\mapsto \frac{f\circ G_i (tx_0)}{t^q}$$ (defined in a neighborhood of $1$) is constant. Put $\psi_x(t):=\frac{f\circ G_i (tx)}{t^q},$ $t\in \mathbb D(1,r):=\{\lambda\in \mathbb C:\ |\lambda-1|<r\},$ where $r$ is sufficiently small.

Assume the contrary, i.e. $\psi_{x_0}$ is non-constant. Then there is $0<r'<r$ such that $y_0\notin \psi_{x_0} (\partial \mathbb D(1,r')).$ Using the uniform convergence argument one can easily see that there is an $\epsilon >0$ and there is a neighborhood $U(x_0)\subset g(V')\setminus\mathcal V$ of $x_0$ such that $\psi_x$ is well defined in a neighborhood of $\overline{\mathbb D(1,r')}$ (decrease $r'$ if necessary) and $\dist (y_0, \psi_{x} (\partial \mathbb D(1,r')))>\epsilon$ whenever $x\in U(x_0).$

Let $V(x_0)$ be an open neighborhood of the point $x_0$ such that $V(x_0)\subset U(x_0)$ and $\dist(y_0, V(y_0)) <\frac{\epsilon}{2},$ where $V(y_0) =f(G_i(V(x_0)))$.

Since $y_0$ lies in the Shilov boundary of $\Omega_1$, there is an $F\in \mathcal O(\Omega_1)\cap \mathcal C (\overline{\Omega_1})$ such that $\max\{|F(x)|:\ x\in V(y_0)\cap \overline{\Omega_1}\}>\max\{|F(x)|:\ x\in  \overline{\Omega_1}\setminus V(y_0)\}$ (otherwise the Shilov boundary of $\Omega_1$ would be contained in $\overline{\Omega_1}\setminus V(y_0)$). Choose $\tilde y\in V(y_0) \cap \overline{\Omega_1}$ at which the maximum on the left side is attained and note that taking $y'\in \Omega_1\cap V(y_0)$ sufficiently close to $\tilde y$ we get the following inequality:
\begin{equation}\label{ee}|F(y')|>\max\{|F(y)|:\ x\in \overline{\Omega_1}\setminus V(y_0)\}.\end{equation}
Let $x'\in V(x_0)$ be such that $y'=f(G_i(x'))$.

First, observe that $m_1(y')=m_1(f(G_i(x')))=m_2(g(G_i(x')))^q = m_2(x')^q,$ so $x'\in \Omega_2$. Note also that $m_1(\psi_{x'}(t))=m_2(x')^q,$ hence $\psi_{x'}(\overline{\mathbb D(1,r')})\subset \Omega_1$. Moreover, $\psi_{x'}(1) = y'$ and $\psi_{x'}(\partial \mathbb D(1,r'))\cap V(y_0)=\varnothing$.

But a function $F\circ \psi_{x'}$ attains its maximum on $\partial \mathbb D(1,r')$. This contradicts (\ref{ee}).

\textit{Step 1''} It is clear that $\mathcal V \subset \{x\in g(V'):\ \Phi(x)=0\}$ for some holomorphic function $\Phi$ on $g(V')$, $\Phi\neq 0$ (the function $\Phi$ may be given explicitly - for example one may take $\Phi(x)=\prod_{j=1}^p \det g'(G_j(x))$ where $G_j$ are local inverses to $g$).

Define $\tilde \Psi (t,x,y):=\prod_{i,j}(f(G_i(x))-t^{p_1} f(H_j(y))),$ $x,y\in g(V')$, $t\in \mathbb D$, where $G_i,$ $H_j$ are local inverses to $G$ defined in a neighborhood of $x$ and $y,$ respectively. Put $\Psi(t,x):= \tilde \Psi (t,t^{p_2}x, x),$ $x\in g(V'),$ $t\in \mathbb D.$ It follows easily from \textit{Step 1'} that for every $x\in \partial_s \Omega_2\setminus \mathcal V$ the mapping $\Psi(\cdot,x)$ vanishes in a neighborhood of $1$. Hence $\Psi(t,x)=0$ for any $t\in \mathbb D$ and $x\in \partial_s\Omega_2\setminus \mathcal V.$ Therefore, for a fixed $t\in \mathbb D$ the mapping $\Phi\cdot \Psi(t,\cdot)$ vanishes on $\partial_s\Omega_2$, so by the properties of the Shilov boundary $\Phi\cdot \Psi\equiv 0.$ Whence $\Psi\equiv 0.$

Fix $x'\in \Omega_2\setminus \mathcal V$, $l\in \{1,\ldots,m\}$ and observe that there is an $i$ such that $f_l(G_i(t^{p_2} x))=t^{p_1} f_l(G_i(x))$ for $t$ in a neighborhood of $1$ and $x$ in a neighborhood of $x'.$ We aim at showing that  \begin{equation}\label{allj} f_l(G_j(t^{p_2} x')) =t^{p_1} f_l(G_j(x'))\quad \text{for}\ j=1,\ldots,p, \ \text{and} \ t\ \text{sufficiently close to}\ 1.\end{equation} To prove it put $y_i=G_i(x')$ and $y_j=G_j(x').$ Note that $y_i$ and $y_j$ may be joined by a path $\gamma:[0,1]\to g^{-1}(\Omega_1)\setminus \mathcal U,$ where $\mathcal U=g^{-1} (\mathcal V).$ Put $\Gamma=g \circ \gamma.$ A standard compactness argument allows us to find a partition of the interval $0=t_0<t_1< \ldots <t_n=1$ and open balls $(B_k)_{k=1}^N$ covering $\Gamma^*$, $B_k\subset \subset \Omega_2\setminus \mathcal V,$ such that $\Gamma([t_{k-1},t_k])\subset B_k$ and preimage $g^{-1}(B_k)$ has exactly $p$ connected components, $k=1,\ldots, N$.

There is a unique holomorphic mapping $H_1$ on $B_1$ such that $g\circ H_1=\id$ and $H_1(\Gamma(t)) = \gamma(t)$ for $t\in [t_0,t_1]$. Note that $H_1=G_i$, so by the identity principle $f_l(H_1(t^{p_2} x))=t^{p_1} f_l(H_1(x))$ for $x\in B_1$ and $t$ sufficiently close to $1$. Similarly, there is a holomorphic mapping $H_2$ on $B_2$ such that $g\circ H_2=\id$, $H_2(\Gamma(t)) = \gamma(t)$ for $t\in [t_1,t_2]$ and $H_1=H_2$ on $B_1\cap B_2$. Using the identity principle again we get the relation $f_l(H_2(t^{t_2} x))=t^{p_1} f_l(H_2(x))$ for $x\in B_2$ and $t$ sufficiently close to $1.$ Proceeding inductively one may construct a mapping $H_N$ holomorphic on $B_N$ such that $H_N=H_{N-1}$ on $B_{N-1}\cap B_N,$ $g\circ H_N=\id$, and $G_N(x')=H_N(\Gamma(t_N))=\gamma(1)=y_j.$ Moreover  $f_l(H_N(t^{p_2} x))=t^{p_1} f_l(H_N(x))$, for $x\in B_N$ and $t$ close to $1.$ Note that $H_N=G_j$ in a neighborhood of $x'$ and this finishes the proof of (\ref{allj}).

Thus, we have shown that for any $x\in \Omega_2\setminus \mathcal V$ the equality $f(G_j(t^{p_2} x)) =t^{p_1} f(G_j(x))$ remains valid for all $j=1,\ldots, p,$ and $t$ sufficiently close to $1.$

\textit{Step 1'{''}} Let us consider the following system of equations $$(\dag)\ \sum_{\sigma\in
\Sigma_p}\prod_{k=1}^p(y_{j_k}-\lambda_{\sigma(k), j_k})=0,\quad \{j_1,\ldots , j_p\} \ \subset \{1,\ldots, m\},\
j_1\leq \ldots \leq j_p,$$ where $\Sigma_p$ denotes the set of $p$-permutations. Note that for
the given $(\lambda_{i,j})_{i=1,\ldots,p}^{j=1,\ldots,m}$ the only solutions $y=(y_1,\ldots, y_m)$ of the system ($\dag$) are given by the formulas $y=(\lambda_{i,1},\ldots, \lambda_{i,m}),$ $i=1,\ldots,p$. To show it observe that any root of the equations in (\dag) with $j_1=\ldots=j_p$ is of the form $(\lambda_{i_1,1},\ldots, \lambda_{i_m,m}).$ What remains to do is to show that $i_1=\ldots=i_m$. If $i_{\iota}\neq i_{\tilde{\iota}}$, then it suffices to analyze the equations in the system (\dag) satisfying $\{j_1,\ldots,j_p\}= \{i_{\iota}, i_{\tilde{\iota}}\}.$ Since these computations are quite simple and tedious, we omit them here.

Multiplying out we get mappings $\xi^I_{\alpha},$ where $|\alpha|<p$ and $I=I(j_1,\ldots,j_p),$ such that $$\sum_{\sigma\in \Sigma_p}\prod_{k=1}^p(y_{j_k}-\lambda_{\sigma(k), j_k})=p! y_{j_1}\ldots y_{j_p}+\sum_{|\alpha|<p}
\xi^I_{\alpha}(\lambda)y^{\alpha}.$$ Observe that $\xi^I_{\alpha}$ are homogenous of order $p-|\alpha|$ and note that they are \emph{quasi-symmetric} in the following sense: \begin{multline}\xi^I_{\alpha} (\lambda_{1,1}, \ldots, \lambda_{1,m}, \ldots, \lambda_{p,1},\ldots, \lambda_{p,m})=\\ \xi^I_{\alpha} (\lambda_{\sigma(1),1}, \ldots, \lambda_{\sigma(1),m},\ldots, \lambda_{\sigma(p),1}, \ldots, \lambda_{\sigma(p),m}) \quad \text{for any}\ \sigma\in \Sigma_p.\end{multline} Therefore it is clear that $\zeta^I_{\alpha}:=\xi^I_{\alpha}\circ f\circ g^{-1}:= \xi^I_{\alpha} (f_1\circ G_1,\ldots, f_m\circ G_1,\ldots, f_1\circ G_p,\ldots, f_m\circ G_p)$ is a well defined holomorphic mapping on $g(V').$

It follows from above considerations (\emph{Step 2}) that \begin{equation}\zeta^I_{\alpha}(t^{p_2}x)=t^{p_1(p-|\alpha|)}\zeta^I_{\alpha}(x)\quad \text{for all}\ x\in \Omega_2\ \text{and}\ t\in \mathbb D.
\end{equation}

Now one may write down the Taylor expansion of $\zeta^I_{\alpha}$ around $0$ in order to verify that $\zeta^I_{\alpha}$ are homogenous polynomials of degree $q(p-|\alpha|)$ (obviously, if $q(p-|\alpha|)\notin \mathbb N,$ then $\zeta^I_{\alpha}\equiv 0$).

Consider the following system of equations: \begin{equation}\label{row}\Theta_I(x,y):=p! y_{j_1}\ldots y_{j_p}+\sum_{|\alpha|<p}
\zeta^I_{\alpha}(x)y^{\alpha}=0,\quad I=I(j_1,\ldots, j_p),\end{equation} $\{j_1,\ldots, j_p\}\subset \{1,\ldots, m\},$ $1\leq j_1\leq\ldots\leq j_p\leq m.$ First observe that \begin{equation}\label{r1}\Theta_I(t^{p_2}x, t^{p_1}y)=t^{pp_1} \Theta_I(x,y),\quad t\in\mathbb C.\end{equation}
Note also that for $x$ lying sufficiently close to $0$ the following property holds:
\begin{equation}\label{r}m_2(x)^q=m_1(y)\quad \text{for any root $y$ the system of equations}\quad \Theta_I(x,\cdot)=0. \end{equation} To prove it take $x\in g(V')$. It follows from the definition of the mappings $\zeta_{\alpha}$ that all roots of the equation (\ref{row}) are given by formulas $y=f(x_i),$ where $g(x_i)=x,$ $i=1,\ldots,p$ (precisely $x_i=g^{-1}(x)$ if $x\notin \mathcal V$). The assumptions of the theorem imply that for such a solution $y$ $$m_1(y)=m_1(f(x_i))=m_2(g(x_i))^q=m_2(x)^q,$$ which proves (\ref{r}) for $x$ sufficiently close to $0$. Making use of (\ref{r1}) we find that the relation (\ref{r}) holds for all $x.$

The equality $\Theta_I(g(x),f(x))=0$ holds in the neighborhood of $0,$ so by the identity principle
$\Theta_I(g(x),f(x))=0$ for $x\in U.$ This means that $f(x)$ is the root of the equations $\Theta_I(g(x),\cdot)=0$ for any $x\in U.$ It follows from (\ref{r}) that $m_1\circ f=(m_2\circ g)^q.$

In order to prove the second assertion it suffices to repeat the above reasoning to the mappings $\xi_{\alpha}^I$ with $I=I(j,\ldots,j),$ $j=1\ldots, m.$  To be more precise let us define
\begin{align}\tilde \xi_k(x):=&\sum_{1\leq i_1<\ldots <i_k\leq p} x_{i_1} \ldots x_{i_k},\quad x= (x_1,\ldots, x_p)\in\mathbb C^p,\\\xi_k(\lambda):=&(\tilde \xi_k(\lambda_1),\ldots, \tilde \xi_k
(\lambda_{m})),\quad \lambda=(\lambda_1,\ldots , \lambda_m)\in (\mathbb C^p)^{m}. \end{align} Put $\zeta_k:=\xi_k\circ f\circ g^{-1}$ and $$\Theta(x,y):=y^p-\zeta_1 (x) y^{p-1} +\ldots +(-1)^p\zeta_p(x).$$ As before we prove that $\Theta(f,g)\equiv 0.$

\textit{Step 3} Now we shall show the theorem for $k>m$. It follows from the Remmert's theorem that $\dim_0 f^{-1}(0)=k-m.$ Using basic properties of analytic sets one can find an $m$-dimensional vector space $L$ in the Grassmannian $\mathbb G(m,k)$ such that $0$ is an isolated point of $L\cap f^{-1}(0)$ and $L\cap g^{-1}(0)$. We lose no generality assuming that the space $L$ is of the form $L=\{(x_1,\ldots, x_m, \sum \alpha^{m+1}_j x_j,\ldots, \sum \alpha^k_j x_j):\ x_i \in \mathbb C\}$ for some $\alpha^l_j\in \mathbb C,$ $j=1,\ldots, m$, $l=m+1,\ldots, k.$ Fix $r>0$ such that the polydisc $(r\mathbb D)^k$ is relatively compact in $V$. Let $\tilde B$ be an arbitrary infinite Blaschke product not vanishing on $\frac{1}{2} \mathbb D$ and define $B(\lambda)=\tilde B(\lambda r^{-1}),$ $\lambda\in r\mathbb D.$

Put $\tilde f:=(f,\psi^{p_1}):=(f,e^{p_1\varphi} ( x_{m+1}-\sum \alpha^{m+1}_j x_j)^{p_1}, \ldots, e^{p_1\varphi} (x_k-\sum \alpha^k_j x_j)^{p_1})$ and $\tilde g:=(g,\psi^{p_2}):=(g,e^{p_2\varphi} (x_{m+1}-\sum \alpha^{m+1}_j x_j)^{p_2}, \ldots, e^{p_2\varphi}( x_k-\sum \alpha^k_j x_j)^{p_2}),$ where $\varphi(x_1,\ldots,x_k):=\frac{1}{B(x_1)}+\ldots + \frac{1}{B(x_k)}$. Observe that the mappings $\tilde f$ and $\tilde g$ are locally open in a~neighborhood of $0$ (as $0$ is an isolated point of the fibers $\tilde f^{-1}(0)$ and $\tilde g^{-1}(0)$).

Put $|y|:=|y_1|+\ldots+ |y_{k-m}|,$ $y\in \mathbb C^{k-m},$ and
$$\nu_i(x,y) :=\left(m_i(x)^{\frac{1}{p_i}} + |y|^{\frac{1}{p_i}}\right)^{p_i},\quad (x,y)\in \mathbb C^k=\mathbb C^m\times \mathbb C^{k-m},\  i=1,2.$$

It is clear that the equality $\nu_1(\tilde f) =\nu_2(\tilde g)^q$ holds in a neighborhood of $0$. Applying the previous step we get a natural number $p$, homogenous polynomials $\tilde\zeta_{\alpha}^I$ and corresponding maps $\tilde\Theta_I$ such that $\tilde \Theta_I( \tilde g, \tilde f)=0$. Moreover, the system of equalities $\tilde \Theta_I(x,y)=0,$ $x,y\in \mathbb C^k,$ implies that $\nu_2(x)^q=\nu_1(y).$

Expanding we infer that $$\tilde\Theta_I (\tilde g, \tilde f)=\tilde \Theta_I ((g,\psi^{p_2}), (f, \psi^{p_1}))=\theta_I(g,f) + e^{\varphi} h_1 + \ldots + e^{s\varphi} h_p$$ for some $s\in \mathbb N$, holomorphic maps $h_i$ on $U$ and a $\theta_I$ given by the formula $\theta_I(x,y):= \tilde \Theta_I((x,0),(y,0)).$ Making use of the construction of $\varphi$ we immediately state that $\theta_I(g,f)\equiv h_1\equiv \ldots\equiv h_p\equiv 0.$ Therefore $\tilde \Theta_I((g,0),(f,0))\equiv 0$. Whence $m_1(f(x))=m_2(g(x))^q$ for all $x\in U,$ as claimed.

The relation (\ref{ex}) may be shown analogously.
\end{proof}

\begin{remark}\label{ber} A consideration of the equality $m_1(f(x))=m_2(x)^p$ in a neighborhood of $0$, where $f$ is a proper holomorphic map and $m_1,$ $m_2$ are Minkowski functionals of pseudoconvex balanced bounded domains is the key point of the proof of the main theorem in \cite{Ber-Piz}. The authors investigated this equality with the help of advanced tools of the projective dynamic.

Note that in \textit{Step 1'} of the proof of Theorem~\ref{main} the more general equality was considered (we did not even need the plurisubharmonicity) and the methods we were using are much simpler (actually, in this case $i=1$ and the other steps of the proof are not needed).
\end{remark}

\begin{remark}
The statement of Theorem~\ref{main} is clear if $m_1$ and $m_2$ are the Euclidean norms and $f$, $g$ are arbitrary holomorphic mappings (as the Euclidean norm is $\mathbb R$-analytic). One may check that in this case $p=1$.

Similarly, the statement of Theorem~\ref{main} is clear in the case when $m_1,$ $m_2$ are operator norms (as the operator norm is $\mathbb R$-analytic except for an analytic set).
\end{remark}

\begin{remark}
Note that in the case when $m=k$ and $q=1$, the number $p$ occurring in the statement of Theorem~\ref{main} is equal to the multiplicity of the mapping $f$ (restricted to some neighborhood of $0$). Note also that for $p=1$ the mappings $f$ and $g$ are not necessary biholomorphic (but then $f=\zeta_1 g$ for a linear mapping $\zeta_1$).

Assume that $p$ occurring in Theorem~\ref{main} is equal to $2.$ Then we are able to solve the equation
(\ref{ex}) and state that $f(x)=Q_1(g(x))+\sqrt{Q_2(g(x))},$ where $Q_1$ is linear mapping, $Q_2$ is a homogenous
polynomial of degree $2,$ and the branch of the square is chosen so that $\sqrt{Q_1\circ g}$ is holomorphic.

Generally, we cannot conjecture that $Q_2$ vanishes. Consider the following example: $m_i(x,y)=|(x,y)|=|x|+|y|,$ $i=1,2,$ $f(x,y) = \frac{1}{2} (x^2+2xy+y^2 ,x^2-2xy+y^2)$ and $g(x,y)= (x^2,y^2).$ Then obviously
$|f(x,y)|=|g(x,y)|,$ $Q_1(x,y)=1/2(x+y,x+y)$ and $Q_2(x,y)=xy.$ \end{remark}

\begin{remark}
The assumptions of the openness of the mappings $f$ and $g$ in a neighborhood of $a$ are important. This is illustrated by the following example: $f(x,y)=(xy,x^2y),$ $g(x,y)= (xy,y)$ and $||(x,y)||=\max\{|x|,|y|\}.$ Clearly $||f(x,y)||=||g(x,y)||$ if and only if $|x|\leq 1$ or $y=0.$

Note also that for any neighborhood $U$ of $0$ the images $f(U)$ and $g(U)$ are not analytic.

It is natural to ask whether the assumption of the openness may be weakened. We would like to point out that answer to this question is obvious in the case $m=2$ - it is sufficient to consider the Weierstrass polynomials of $f$ and $g$. This reasoning however cannot be applied to $m\geq 3.$
\end{remark}

\section{Quasi-circular domains}
Let $k_1,\ldots,k_n$ be natural numbers. A domain $D$ of $\mathbb C^n$ is said to be $(k_1,\ldots,k_n)$-circular if \begin{equation}\label{defcir} (\lambda^{k_1}x_1,\ldots,\lambda^{k_n}x_n) \in D\quad \text{whenever}\quad \lambda\in\partial\mathbb D,\ x=(x_1,\ldots,x_n)\in D.\end{equation} If the formula (\ref{defcir}) holds for any $\lambda\in \overline{\mathbb D},$ then $D$ is said to be $(k_1,\ldots,k_n)$-balanced (or $(k_1,\ldots,k_n)$-complete circular).

A domain $\Omega$ is called to be \textit{quasi-circular} (respectively \textit{quasi-balanced}) if it is $k$-circular (resp. $k$-balanced) for some $k=(k_1,\ldots,k_n)\in \mathbb N^n.$

For $k=(k_1,\ldots,k_n)$-balanced domain $D\subset \mathbb C^n$ one may define its \emph{$k$-Minkowski functional} (a \emph{quasi}-Minkowski functional) by the following formula
\begin{equation}\mu_{D,k}(x):=\inf\{\lambda>0:\ (\lambda^{-k_1}x_1,\ldots,\lambda^{-k_n}x_n)\in D\},
\end{equation} $x=(x_1,\ldots,x_m)\in\mathbb C^n.$ The introduced above function has similar properties as the standard Minkowski functional.  Recall them for the convenience of the reader:

\begin{align} &\mu_{D,k}(\alpha^{k_1}x_1,\ldots,\alpha^{k_n}x_n)=|\alpha|\mu_{D,k}(x),\quad x\in\mathbb C^n,\ \alpha\in\mathbb C,\\
&D=\{x\in\mathbb C^n:\ \mu_{D,k}(x)<1\}.
\end{align}

For $k=(k_1,\ldots,k_n)\in \mathbb N$ and $x\in \mathbb C^n$ denote $k.x:=(x_1^{k_1},\ldots , x_n^{k_n}).$

Let $D$ be a $k$-balanced domain and $\mu_{D,k}$ be the quasi-Minkowski functional associated with this domain. Put $\tilde k_j:=\frac{k_1\cdot \ldots \cdot k_n}{k_j},$ $\tilde k:=(\tilde k_1,\ldots,\tilde k_n)$ and define $m(x):=\mu_{D,k} (\tilde k^{-1}.x)$. One may check that $m$ is radial. In particular, $m$ is the Minkowski functional of a bounded balanced domain and it satisfies the property $m(\tilde k.x)=\mu_{D,k}(x),$ $x\in \mathbb C^n.$ On the other hand $\tilde k.f$ is open provided that $f$ is an open holomorphic mapping.

This simple observation leads us to the following
\begin{corollary}\label{maincor} Let $\mu_1,$ $\mu_2$ be quasi-Minkowski functionals of quasi-balanced domains. Let $f,g:U\rightarrow \mathbb C^{m}$ be a holomorphic mapping such that $f(a)=g(a)=0,$ for some $a\in U\subset \mathbb C^k,$ $k\geq m.$ Assume that $q\in\mathbb R.$ If $\mu_1(f(x))= (\mu_2 (g(x)))^q $ in a neighborhood $V\subset U$ of $a$ and the restrictions $f|_{V'}$, $g|_{V'}$ are open, then $\mu_1\circ f(x)=(\mu_2 \circ g(x))^q$ for all $x\in U$ and $q\in \mathbb Q_{>0}$.
\end{corollary}

One can try to derive a counterpart of the second assertion of Theorem~\ref{main} in the case of quasi-Minkowski functionals. Since the possible formula is a little complicated and self-evident, we omit it here.

\section{Applications to the paper of Boutat}

\begin{remark} It is well known by the Bell's result (see \cite{Bel}), that any proper mapping $f$ between complete
quasi-circular domain such that $f^{-1}(0)=\{0\},$ is a polynomial. So we may expand $f=\sum_{j=p}^q Q_j,$
$p\leq q,$ where $Q_j$ are homogenous of degree $j.$ Let us introduce the following notation: $\rho(f):=Q_p,$
$\varrho(f):=Q_q.$ \end{remark}

Define $||x||_k:=\sum |x_i|^{\frac{1}{k_i}}.$

\begin{proposition}\label{p} Let $D,\Omega_1,\Omega_2\subset\subset \mathbb C^n$ be pseudoconvex quasi-balanced domains ($\Omega_1$ and $\Omega_2$ are assumed to be $k$ and $l$ balanced, respectively). Let $f_i:D\to \Omega_i$ be proper
holomorphic mappings such that $f_i^{-1}(0)=\{0\},$ $i=1,2.$  Assume that there are $m,M>0$ such that $m ||f_2(x)||_{l}^q \leq ||f_1(x)||_{k} \leq M||f_2(x)||_{l}^q.$

Then $\mu_1(f_1(x))=\mu_2(f_2(x))^q,$ $\mu_1(\varrho(f_1)(x)) =\mu_2(\varrho(f_2)(x))^q$ and $\mu_1(\rho(f_1)(x))= \mu_2(\rho(f_2)(x))^q$ for $x\in\mathbb C^n,$ where $\mu_1$ and $\mu_2$ are the $k$- and $l$-Minkowski functionals of $\Omega_1,$ $\Omega_2,$ respectively.

In particular, if $f_1$ is a homogenous polynomial, then $f_2$ is homogenous, as well.
\end{proposition}

\begin{proof} Considering instead of $f_1$ and $f_2$ the mappings $\tilde k. f_1$ and $\tilde l. f_2,$ where $\tilde k_j=k_1\ldots k_n k_j^{-1},$ and $\tilde l_j=l_1\ldots l_n l_j^{-1},$ we may restrict ourselves to the case of balanced domains (i.e. $k_j=l_j=1$, $j=1,\ldots, n$). It is well known that $g_{\Omega_1}(0,f_1(x))=q g_{\Omega_2}(0,f_2(x)).$
Therefore $\mu_1(f_1(x))=\mu_2(f_2(x))^q$ for $x\in \Omega.$ Applying Corollary~\ref{maincor} we state that $\mu_1(f_1(x))=\mu_2(f_2(x))^q$ for $x\in\mathbb C^n.$

Write $f_1=\sum_{j=n_1}^{n_2} Q_j,$ where $Q_j$ is a homogenous polynomial of degree $j.$ Considering
the values of the equations $t^{-n_1} \mu_1(f_1(tx)) = t^{-n_1} \mu_2(f_2(tx))^q$ and $t^{n_2}\mu _1(f_1(x/t)) = t^{n_2} \mu_2(f_2(x/t))^q$ at $t=0$ we easily get the assertion. \end{proof}

\begin{example}[Generalization of the main result of \cite{Bou}]\label{gen} Let $\Omega_1$ be a bounded complete $k$-circular domain and $\Omega_2$ a bounded balanced domain. Suppose that $f:\Omega_1\to \Omega_2$ is a proper mapping such
that $f^{-1}(0)=\{0\}.$ Let $f=\sum_{j\geq p} f_j,$ where $f_j$ is $k$-homogenous of order $j$ (i.e.
$f_j(t^{k_1}x_1,\ldots, t^{k_n}x_n)=t^jf(x),$ $x\in \Omega_1,$ $t\in \mathbb D$). Assume that $f_p^{-1}(0)=0.$ Then
$f=f_p.$ \end{example}

\begin{proof} It is clear that $\Omega_1$, $\Omega_2$ and $D$ may be assumed to be pseudoconvex. One may easily check that
$A||x||^p_k\leq ||f_p(x)||\leq B||x||^p_k$ for some positive $A,$ $B$ (use the fact that $f_p ( x_1 ||x||_k^{-k_1}, \ldots , x_n ||x||_k^{-k_n})$ is uniformly bounded for $x\neq 0$). This implies that $m||x||_k^p\leq ||f(x)||\leq M||x||_k^p,$
for some constants $m,M>0.$ Now it suffices to apply Proposition~\ref{p} to get that $f$ is homogenous. \end{proof}

\end{document}